\documentclass{amsart}
\usepackage{amsmath,amssymb,amsthm}
\usepackage{array}
\newcolumntype{C}[1]{>{\centering\arraybackslash}p{#1}}
\usepackage{amssymb}
\newcommand{\p}{^{\prime}}

\newtheorem{thm}{Theorem}
\newtheorem{prop}{Proposition}
\newtheorem{lemma}{Lemma}
\newtheorem{cor}{Corollary}

\title{Anti-Palindromic Compositions}
\author{George E. Andrews, Matthew Just, and Greg Simay}
\date{}

\begin{document}

\maketitle

\begin{abstract}
    A palindromic composition of $n$ is a composition of $n$ which can be read the same way forwards and backwards. In this paper we define an anti-palindromic composition of $n$ to be a composition of $n$ which has no mirror symmetry amongst its parts. We then give a surprising connection between the number of anti-palindromic compositions of $n$ and the so-called tribonacci sequence, a generalization of the Fibonacci sequence. We conclude by defining a new $q$-analogue of the Fibonacci sequence, which is related to certain equivalence classes of anti-palindromic compositions. 
\end{abstract}

%\begin{abstract}
 %    In recent work, Andrews, Benjamin, and Simay  prove by algebraic methods a surprising connection between anti-palindromic compositions in combinatorics, and the somewhat esoteric ``tribonacci" sequence $f_3(n)$ defined by a similar recursion to Fibonacci numbers. They show that the number of anti-palindromic compositions of $n$ is equal to $2 \cdot f_3(n-2)$. In this note, we present a combinatorial proof of this identity. 
%\end{abstract}

\section{Introduction}

Let $\sigma=(\sigma_1,\sigma_2,\ldots,\sigma_s)$ be a sequence of positive integers such that $\sum \sigma_i = n$. The sequence $\sigma$ is called a \textit{composition} of $n$ of length $s$. The numbers $\sigma_i$ are called the \textit{parts} of the composition. The number of compositions of $n$ equals $2^{n-1}$, while the number of compositions of $n$ into $s$ parts equals ${n-1\choose s-1}$. The empty composition is often considered the only composition of 0, having length equal to 0.

\subsection{Palindromic and anti-palindromic compositions}

If $\sigma_i=\sigma_{s-i+1}$ for all $i$, then $\sigma$ is called a \textit{palindromic} composition. It is well known \cite{HB75} that if $pc(n)$ is the number of palindromic compositions, then $pc(n)=2^{\lfloor \frac{n}{2}\rfloor}$. For instance, the $pc(5)=4$ palindromic compositions of $5$ are \[(5), \ \ (1,3,1), \  \ (2,1,2), \  \ \text{and} \  \ (1,1,1,1,1).\] 
Recent work of the authors \cite{AS20, J20} generalize this result to compositions that are palindromic modulo $m$, where the condition $\sigma_i=\sigma_{s-i+1}$ is replaced with the weaker condition $\sigma_{i}\equiv \sigma_{s-i+1}$ (mod $m$).

If $\sigma_{i}\neq \sigma_{s-i+1}$ for all $i\neq \frac{s+1}{2}$, then we say $\sigma$ is an \textit{anti-palindromic} composition. Let $ac(n)$ be the number of anti-palindromic compositions of $n$. Then the $ac(4)=5$ anti-palindromic compositions of $4$ are \[(4), \ \ (1,3), \ \ (3,1), \ \ (1,1,2), \ \ \text{and} \ \ (2,1,1).\] Furthermore, let $ac(n,s)$ be the number of anti-palindromic compositions of $n$ of length $s$, $ac_0(n)$ be the number of anti-palindromic compositions of $n$ of even length, and $ac_1(n)$ be the number of anti-palindromic compositions of odd length (thus $ac(n)=ac_0(n)+ac_1(n)$). 

Notice that for each anti-palindromic composition of $n$ of length $s$, we can form $2^{\lfloor \frac{s}{2} \rfloor}$ \textit{flip-equivalent} anti-palindromic compositions of $n$ of length $s$ by switching any number of the pairs $\sigma_i$ and $\sigma_{s-i+1}$ ($i\neq \frac{s+1}{2}$). For instance, the anti-palindromic compositions
\[(1,3,3,2,4), \ \ (1,2,3,3,4), \ \  (4,3,3,2,1), \ \ \text{and} \ \ (4,2,3,3,1)\]
are all flip-equivalent compositions of 13 of length 5. The sets of flip-equivalent anti-palindromic compositions of $n$ form a partition of the set of all anti-palindromic compositions of $n$, and we refer to each equivalence class as a \textit{reduced anti-palindromic composition} of $n$ of length $s$. Let $rac(n)$ equal the total number of reduced anti-palindromic compositions of $n$, and $rac(n,s)$ equal the number of reduced anti-palindromic compositions of $n$ of length $s$. Furthermore, let $rac_0(n)$ and $rac_1(n)$ equal the total number of even and odd reduced anti-palindromic compositions of $n$, respectively. Clearly we have $rac(n)=rac_0(n)+rac_1(n)$. Since each equivalence class contains $2^{\lfloor \frac{s}{2}\rfloor}$ anti-palindromic compositions, it follows that
\[rac(n,s)=\frac{ac(n,s)}{2^{\lfloor \frac{s}{2}\rfloor}}.\]

Our primary results regarding the formulae for these functions come from observations made in Table \ref{tab1} and Table \ref{tab2}.

\begin{table}[]
    \centering
    \begin{tabular}{c|cccccc}
        $n$     & $ac_0(n)$   & $ac_1(n)$     & $ac(n)$  & $rac_0(n)$ & $rac_1(n)$ & $rac(n)$ \\ 
         \hline 
         $0$  & 1  & 0   & 1   &1&0& 1 \\
         $1$  & 0  & 1   & 1   &0&1& 1   \\
         $2$  & 0  & 1   & 1   &0&1& 1\\
         $3$  & 2  & 1   & 3   &1&1& 2\\
         $4$  & 2  & 3   & 5   &1&2& 3\\
         $5$  & 4  & 5   & 9   &2&3& 5\\
         $6$  & 8  & 9   & 17  &3&5& 8\\
         $7$  & 14 & 17  & 31  &5&8& 13\\
         $8$  & 26 & 31  & 57  &8&13& 21\\
         $9$  & 48 & 57  & 105 &13&21& 34\\
         $10$ & 88 & 105 & 193 &21&34& 55
    \end{tabular}
    \caption{Values of $ac_0(n)$, $ac_1(n)$, $ac(n)$, $rac_0(n)$, $rac_1(n)$, and $rac(n)$ for $n\leq 10$.}
    \label{tab1}
\end{table}

\begin{table}[]
    \centering
    \begin{tabular}{c|C{.5in} C{.5in} C{.5in} C{.5in} C{.5in} C{.5in}}
         $n$&$ac(n,0)$&$ac(n,1)$&$ac(n,2)$&$ac(n,3)$&$ac(n,4)$&$ac(n,5)$\\
         \hline
         0& 1&0&0&0&0&0\\
         1& 0&1&0&0&0&0\\
         2& 0&1&0&0&0&0\\
         3& 0&1&2&0&0&0\\
         4& 0&1&2&2&0&0\\
         5& 0&1&4&4&0&0\\
         6& 0&1&4&8&4&0\\
         7& 0&1&6&12&8&4\\
         8& 0&1&6&18&20&12
    \end{tabular}
    
    \vspace{.2in}
    
    \begin{tabular}{c|C{.5in}C{.5in}C{.5in}C{.5in}C{.5in}C{.5in}}
         $n$&$rac(n,0)$&$rac(n,1)$&$rac(n,2)$&$rac(n,3)$&$rac(n,4)$&$rac(n,5)$\\
         \hline
         0& 1&0&0&0&0&0\\
         1& 0&1&0&0&0&0\\
         2& 0&1&0&0&0&0\\
         3& 0&1&1&0&0&0\\
         4& 0&1&1&1&0&0\\
         5& 0&1&2&2&0&0\\
         6& 0&1&2&4&1&0\\
         7& 0&1&3&6&2&1\\
         8& 0&1&3&9&5&3
    \end{tabular}
    
    \caption{Values of $ac(n,s)$ (top) and $rac(n,s)$ (bottom) for $n\leq 8$.}
    \label{tab2}
\end{table}

\subsection{The $k$-binocci numbers}

Recall the $n$th Fibonacci numer is given by $f_2(n)=0$ for $n<1$, $f_2(1)=1$, and $f_2(n)=f_2(n-1)+f_2(n-2)$ for all $n\geq 2$\footnote{In some applications the offset $f_2(0)=f_2(1)=1$ is used.}. The $n$th \textit{tribonacci} number is given by $f_3(n)=0$ for $n<1$, $f_3(1)=1$, and $f_3(n)=f_3(n-1)+f_3(n-2)+f_3(n-3)$ for $n\geq 2$. The sequence begins \[0,1,1,2,4,7,13,24,\ldots\] see OEIS A000073 \cite{trib}. It has been suggested tribonacci numbers appear in Darwin's \textit{Origins of Species} in a similar relation to elephant population growth as Fibonacci numbers bear to rabbit populations \cite{eleph}. In general, we can define the $n$th $k$-binocci number by $f_k(n)=0$ for $n<1$, $f_k(1)=1$, and \[f_k(n) = \sum_{i=1}^k f_k(n-i)\] for $n>1$. Connections between $k$-bonacci numbers for various $k$ have been studied by Bravo and Luca \cite{BL13}. In a paper of Benjamin, Chinn, Scott, and Simay \cite{BCSS11}, formulae for the $k$-binocci are developed. For instance, we have
\[f_3(n+1)=\sum_{j=0}^{\lfloor n/4\rfloor} (-1)^j {n-3j\choose j} \frac{n-2j}{n-3j}2^{n-4j-1}.\]
The $k$-binocci numbers also play a role in computing the probability of flipping exactly $k$ consecutive heads in $n$ flips of a fair coin.

\subsection{Formulae for anti-palindromic compositions}

Our first result gives a surprising connection between the tribonacci numbers and anti-palindromic compositions of even length. 

\begin{thm}
    For all $n\geq 1$ we have $ac_0(n)=2\cdot f_3(n-2)$.
\end{thm}

This theorem can be deduced by a careful inspection of the identity \[\frac{\left( \frac{q}{1-q}\right)^2- \frac{q^2}{1-q^2}}{1-\left[\left( \frac{q}{1-q}\right)^2- \frac{q^2}{1-q^2} \right]} = \frac{2q^3}{1-(q+q^2+q^3)}.\] Indeed, the left hand side is \[\sum_{n\geq1} ac_0(n) q^n\] since every even-length anti-palindromic composition is a sequence of pairs of distinct positive integers, and the right hand side is \[\sum_{n\geq 1}2\cdot f_3(n)q^{n+2}.\] We will give an algebraic (Section 2.1) and combinatorial (Section 2.2) proof of this result. Note that $ac(0)=ac_0(0)=1$, as the empty composition is vacuously anti-palindromic. Our next result gives the number of anti-palindromic compositions of $n$.

\begin{thm}
    For all $n\geq 1$,
    \[ac(n) = f_3(n)+f_3(n-2).\]
\end{thm}
We will prove Theorem 2 in Section 2.3, and also observe that (for $n\geq 2$)
\[ac_1(n)=f_3(n-1)+f_3(n-3).\]

In Section 2.4, we prove the following results which give the formulae for $ac(n,s)$.
\begin{thm}
    Let $s\geq 0$ be a fixed integer, and 
    \[G(q,s) = \sum_{n\geq 0} ac(n,s) q^n.\]
    Then for $|q|<1$
    \[G(q,s) = \frac{2^{\lfloor s/2\rfloor} q^{\lfloor 3s/2\rfloor} }{(1-q)^s  (1+q)^{\lfloor s/2\rfloor}}.\]
\end{thm}
For instance, $G(q,0)=1$,
\[G(q,1)=\frac{q}{1-q}=q+q^2+q^3+\ldots,\]
and
\[G(q,2)=\frac{2q^3}{1-q-q^2+q^3}=2q^3+2q^4+4q^5+\ldots,\]
which give the (verifiable) formulae $ac(0,0)=1$, $ac(n,0)=0$ for $n>0$, $ac(n,1)=1$ for $n>0$, and $ac(n,2)=2\cdot\lfloor \frac{n-1}{2}\rfloor$ for $n>1$. For $s\geq 2$, we have the following corollary.

\begin{cor}
    Let $a$ be a positive integer. If $s=2a$, then
    \[ac(n,s)=\sum_{r+2t=n-3a} 2^a{a+r-1\choose r}{a+t-1\choose t}, \]
    and if $s=2a+1$, then
    \[ac(n,s)=\sum_{r+2t=n-3a} 2^a{a+r\choose r}{a+t-1\choose t}, \]
\end{cor}

By the observation in Section 1.1 regarding $rac(n,s)$ and $ac(n,s)$, we also have a formula for $rac(n,s)$ by dividing by the appropriate power of 2.

\subsection{An observation regarding the Fibonacci numbers}
Recall the following two $q$-analogues of the Fibonacci numbers, 

\[F_n(q)=\begin{cases}
    0 & n=0, \\
    1 & n=1, \\
    F_{n-1}(q) + q^{n-2}F_{n-2}(q)& n>1,
\end{cases}
\]
and
\[\hat{F}_n(q)=\begin{cases}
    0 & n=0, \\
    1 & n=1, \\
    \hat{F}_{n-1}(q) + q^{n-1}\hat{F}_{n-2}(q)& n>1.
\end{cases}
\]
These are referred to as $q$-analogues due to the property that $F_n(q)\rightarrow f_2(n)$ and $\hat{F}_n(q)\rightarrow f_2(n)$ as $q\rightarrow 1^-$.
Properties of these two sequences of polynomials have been studied extensively, see for instance \cite{A04,C75,C03,P06}.

We define a new $q$-analogue of the Fibonacci numbers, which will have a connection to the anti-palindromic compositions. Define
\[\phi_n(q) = \begin{cases}
    q & n=1,\\
    q & n=2,\\
    q+q^2 & n=3,\\
    \phi_{n-1}(q)+\phi_{n-2}(q)+(q^2-1)\phi_{n-3}(q) & n>3.
\end{cases}
\]
Clearly $\phi_n(q)\rightarrow f_2(n)$ as $q\rightarrow 1^-$ for all $n\geq 1$, and our final result gives a combinatorial description of the coefficients of these polynomials. For convenience, we set $\phi_0(q)=1$.

\begin{thm}
    The coefficient of $q^s$ in the polynomial $\phi_n(q)$ equals $rac(n,s)$.
\end{thm}

We will give a proof of Theorem 4 in Section 2.5. The first few polynomials $\phi_n(q)$ are given below, where the coefficients can be compared with Table \ref{tab2}.

\begin{align*}
    \phi_0(q) & = 1 & \phi_5(q)&= q+2q^2+2q^3\\
    \phi_1(q) & = q  & \phi_6(q) &= q+2q^2+4q^3+q^4\\
    \phi_2(q) & = q  & \phi_7(q) &= q+3q^2+6q^3+2q^4+q^5\\
    \phi_3(q) & = q+q^2  & \phi_8(q) &= q+3q^2+9q^3+5q^4+3q^5\\
    \phi_4(q) & =  q+q^2+q^3 & \phi_9(q) &= q+4q^2+12q^3+8q^4+8q^5+q^6.
\end{align*}
Also in Section 2.5, we deduce the following corollary.
\begin{cor}
    For $n\geq 1$ we have $rac_0(n)=f_2(n-2)$, $rac_1(n)=f_2(n-1)$, and $rac(n)=f_2(n)$.
\end{cor}

We summarize our results regarding $ac(n)$ and $rac(n)$ for sufficiently large $n$ below, illustrating the elegance of the formulae.

\begin{align*}
    ac_0(n)&=2\cdot f_3(n-2) & rac_0(n)&=f_2(n-2)\\
    ac_1(n)&=f_3(n-1)+f_3(n-3) & rac_1(n)&=f_2(n-1)\\
    ac(n)&=f_3(n)+f_3(n-2) & rac(n)&=f_2(n).
\end{align*}

\section{Proofs of theorems}

\subsection{Algebraic proof of Theorem 1}

In light of the fact that $ac_0(1)=ac_0(2)=0$, $ac_0(3)=ac_0(4)=2$, and $ac_0(5)=4$, we see that the theorem is true for $n<6$. Assume now $n\geq 6$. Clearly we can construct an anti-palindromic composition of $n$ from one of two fewer parts by inserting $j$ at the beginning and $k$ at the end (making sure $j\neq k$), where if the inner composition is a composition of $m$, then $j+k$ must equal $m-n$. Hence 
\[ac_0(n) = \sum_{m=0}^{n-3}\left(n-m-1-\chi(n-m)\right)ac_0(m), \]
where $\chi(j)=1$ if $j$ is even and $0$ if $j$ is odd. The term $\left(n-m-1-\chi(n-m)\right)$ accounts for the number of $j$ and $k$.
Hence 
\begin{align*}
ac_0(n)-ac_0(n-1)=&  \sum_{m=0}^{n-3}\left(n-m-1-\chi(n-m)\right)ac_0(m) \\
&- \sum_{m=0}^{n-4}\left(n-1-m-1-\chi(n-1-m)\right)ac_0(m)\\
=&\left(2-\chi(3) \right)ac_0(n-3) +\sum_{m=0}^{n-4}\left(n-m-1-\chi(n-m)\right)ac_0(m)\\ &-\sum_{m=0}^{n-4}\left(n-1-m-1-\chi(n-1-m)\right)ac_0(m) \\
=& 2ac_0(n-3) + 2\sum_{m=0}^{n-4}\chi(n-m-1)ac_0(m).
\end{align*}
Thus
\[ac_0(n)-ac_0(n-1)-2ac_0(n-3)=2\sum_{m=0}^{n-4} \chi(n-m-1)ac_0(m).\]
Let $r(n)=ac_0(n)-ac_0(n-1)-2ac_0(n-3)$. Then
\begin{align*}
   r(n)+r(n-1)=& 2\sum_{m=0}^{n-4}\chi(n-m-1)ac_0(m)\\
   &+2\sum_{m=0}^{n-5}\chi(n-m-2)ac_0(m)\\
   =& 2\sum_{m=0}^{n-5}ac_0(m)
\end{align*}
since $\chi(n)+\chi(n-1)=1$ and $\chi(3)=0$. Therefore,
\[r(n)+r(n-1)-\left(r(n-1)+r(n-2) \right) = 2ac_0(n-5),\]
and simplifying we obtain
\[ac_0(n)-m_1(n-1)-ac_0(n-2)-ac_0(n-3)=0.\]
This is the defining recurrence for $f_3(n)$, and since $2\cdot f_3(n-2)=ac_0(n)$ for $n>6$, we see that by induction $ac_0(n)=2\cdot f_3(n-2)$ for all $n\geq 1$.

\subsection{Combinatorial proof of Theorem 1}\label{comb}

We begin with a lemma regarding the tribonacci numbers.

\begin{lemma}
    For $n\geq 2$, the tribonacci number $f_3(n)$ equals number of compositions of $n-1$ with parts equal to 1, 2, or 3.
\end{lemma}

\begin{proof}
     First note that $f_3(2)=1$, $f_3(3)=2$, and $f_3(4)=4$. Since the compositions of $1$, $2$, and $3$ only consist of parts equal to 1, 2, or 3, and the number of compositions of $n$ is equal to $2^{n-1}$, the lemma holds for $n\leq 4$. Now for $n>4$, each composition of $n-1$ into parts equal to 1, 2, or 3 is formed by taking a composition of $n-4$, $n-3$, or $n-2$ and adjoining a 3, 2, or 1, respectively. Thus the number of compositions of $n-1$ into parts equal to 1, 2, or 3 is equal to $f_3(n-3)+f_3(n-2)+f_3(n-1)=f_3(n)$. \end{proof}
     
     We will now show that for $n\geq 3$ the number of compositions of $n-3$ into parts equal to 1, 2, or 3 equals the number of anti-palindromic compositions of $n$. Since $ac_0(1)=0=2\cdot f_3(-1)$ and $ac_0(2)=0=2\cdot f_3(0)$ this will establish the theorem.

\begin{proof}[Proof of Theorem 1]
     For $n=2$ we see that $ac(2)=2\cdot f_3(0) = 0$, so for any $n\geq 3$ start with a composition $\sigma$ of $n-3$ into parts equal to 1, 2, or 3. The key will be to use $\sigma$ to construct a sequence of pairs of distinct positive integers with sum equal to $n$.
    
    Now recall a \textit{partition} of $n$ is a composition of $n$ where the parts are written in non-increasing order. Let $\sigma+\tau$ denote sequence concatenation, as in $(1,2)+(4,5)=(1,2,4,5)$. For our choice of $\sigma$, we can find partitions $\lambda_1$, $\lambda_2$, \ldots, $\lambda_r$ with parts equal to 1 or 2 (or the empty partition, $\varnothing$) such that \[\sigma = \lambda_1+\sigma_2+\lambda_2\ldots+\sigma_r+\lambda_r,\] where each $\sigma_j$ is either equal to the composition $(3)$ or equal to the composition $(1,2)$. 
    
    For example, take the composition \[\sigma=(2,3,1,1,2,2,1,1,1,2,1,3)\] of 20.  Then we can decompose $\sigma$ as 
    \begin{align*}
       \lambda_1&=(2)\\
       \sigma_2&=(3)\\
       \lambda_2&=(1)\\
       \sigma_3&=(1,2)\\
       \lambda_3&=(2,1,1)\\
       \sigma_4&=(1,2)\\
       \lambda_4&=(1)\\
       \sigma_5&=(3)\\
       \lambda_5&=\varnothing.
    \end{align*}
    It is not difficult to see that this decomposition is unique; the only way a segment in the composition that is a partition with parts equal to 1 or 2 terminates is with the segment $(3)$ or the segment $(1,2)$.
    
    Now given the decomposition $\sigma=\lambda_1+\sigma_2+\lambda_2+\ldots+\sigma_r+\lambda_r$, form a sequence of pairs $(s_1,\lambda_1)$, $(s_2,\lambda_2)$, \ldots, $(s_r,\lambda_r)$ where $s_1=+3$, $s_j=+3$ if $\sigma_j=(3)$, and $s_j=-3$ if $\sigma_j=(1,2)$. For our example shown above, we have the pairs
    \[(+3,(2)), \ (+3,(1)), \ (-3,(2,1,1)), \ (-3,(1)), \ (+3,\varnothing).\]

    For each pair $(s_j,\lambda_j)$ we now form a new pair $(b_j,c_j)$ in the following way. Start with $b_j=2$ and $c_j=1$. For each 2 in the partition $\lambda_j$ increase both $b_j$ and $c_j$ by one. For each 1 in the partition $\lambda_j$ increase $b_j$ by one. We now have pairs $(b_j,c_j)$ of positive integers such that $b_j>c_j$. Now if $s_j=+3$ we are done. If $s_j=-3$, we switch the numerical values of $b_j$ and $c_j$ so that $b_j<c_j$, and then we are done. 
    
    Finally form the anti-palindromic composition $\tau=(\tau_1,\tau_2,\ldots,\tau_{2r})$ by setting $\tau_j=b_j$ and $\tau_{2r-j+1}=c_j$. Notice that though we started with a composition of $n-3$ this is a composition of $n$; the addition of 3 came from inserting $s_1=+3$. In our toy example we have \[\tau=(3,3,2,1,2,1,3,5,1,2).\]
    
    We have now embedded the compositions of $n-3$ made up of parts equal to 1, 2, or 3 into the anti-palindromic compositions of $n$. We still need to embed a second, disjoint copy. To do this we return to the pairs $(s_j,\lambda_j)$ and make a new collection of pairs $(s\p_j,\lambda_j)$ by setting $s\p_j=-s_j$. Now following the same procedure as before we construct an anti-palindromic word $\tau\p$ that is, in fact, the \textit{reverse} of $\tau$. Again looking at our example from before we have \[\tau\p = (2,1,5,3,1,2,1,2,3,3).\]
    
    To show that these two embedded sets are disjoint, notice that for a composition $\tau$ formed by using $s_1=+3$ we have $\tau_1>\tau_{2r}$, and that for a word $\tau\p$ formed by using $s_1=-3$ we have $\tau_1<\tau_{2r}$.
    
    Showing this process reverses and that we can send the pairs $\{\tau,\tau\p\}$ of an anti-palindromic composition of $n$ and its reverse back to a composition of $n-3$ with parts equal to 1, 2, or 3 is straightforward, which the reader can verify. \end{proof}
    
    \subsection{Proof of Theorem 2}
    
    In this section we develop the formula for $ac(n)$. We start by proving some initial observations regarding $ac_0(n)$, $ac_1(n)$, $ac(n)$, and $ac(n,s)$.
    
    \begin{prop}
        For all $n\geq3$, we have
        \[ac_0(n) = f_3(n-1) + f_3(n-5).\]
    \end{prop}
    
    \begin{proof}
        This is just two applications of the defining recurrence for $f_3(n)$, recalling that $f_3(n)=0$ for $n<1$.
        \begin{align*}
            f_3(n-1)+f_3(n-5)&=f_3(n-2)+f_3(n-3)+f_3(n-4)+f_3(n-5) \\
            &= f_3(n-2) + f_3(n-2) \\
            &= 2\cdot f_3(n-2)\\
            &=ac_0(n)\qedhere
        \end{align*}
    \end{proof}
    
    \begin{prop}\label{key}
        We have $ac(0,0)=1$, $ac(0,1)=0$, and for all $n\geq 0$ and $s\geq 0$
        \[ac(n,2s)+ac(n,2s+1)=ac(n+1,2s+1).\]
    \end{prop}
    
    \begin{proof}
        When $n=0$, there is only one composition (the empty composition) which has length 0.
        
        Now any anti-palindromic composition $\sigma$ of $n+1$ of length $2s+1$ has a central part $\sigma_{s+1}$. If $\sigma_{s+1}=1$, this composition can be formed from an anti-palindromic composition of $n$ of length $2s$ by adding a central part equal of 1. If $\sigma_{s+1}>1$, this composition can be formed from an anti-palindromic composition of $n$ of length $2s+1$ by adding 1 to the central part. Therefore, $ac(n,2s)+ac(n,2s+1)=ac(n+1,2s+1)$. \end{proof}
        
        \begin{prop}\label{sum1}
            For $n\geq 0$ and $s\geq0$
            \[ac(n,2s+1)=\sum_{j=0}^{n-1}ac(j,2s),\]
            where in the case $n=0$ we take the empty sum to be 0.
        \end{prop}
        
        \begin{proof}
            Let $n>0$. Then by applying Proposition \ref{key} $n$ times, we have
           \begin{align*} ac(n,2s+1)&=ac(n-1,2s+1)+ac(n-1,2s) \\
           &=ac(n-2,2s+1)+ac(n-2,2s)+ac(n-1,2s)\\
           & \ \ \vdots\\
           &=ac(0,2s+1) + \sum_{j=0}^{n-1}ac(j,2s).
           \end{align*}
           Since $ac(0,2s+1)=0$ for all $s\geq 0$, the result follows. \end{proof}

        \begin{prop} \label{back}
            For all $n\geq 0$ \[ac(n)=ac_1(n+1).\]
        \end{prop}
        
        \begin{proof}
            If $n=0$, we see that $ac(0)=ac_1(1)=1$. If $n>0$, by definition we have
            \begin{align*}
                ac(n)&=\sum_{s\geq 0}ac(n,s) \\
                &= \sum_{j\geq 0 }\left(ac(n,2j )+ac(n,2j+1)\right) \\
                &=\sum_{j\geq 0} ac(n+1,2j+1)
            \end{align*}
        by Proposition \ref{key}. But this last expression is equal to $ac_1(n+1)$. \end{proof}
        
        \begin{prop}\label{sum2}
            For $n\geq0$
            \[ac_1(n)=\sum_{j=0}^{n-1}ac_0(j),\] where in the case $n=0$ we take the empty sum to be 0.
        \end{prop}
        
        \begin{proof}
            For $n>0$, we have by Proposition \ref{back}
            \begin{align*}
                ac_1(n)&=ac(n-1) \\
                &=ac_0(n-1)+ac_1(n-1) .
            \end{align*}
            Now if $n=1$, we are done since $ac_1(0)=0$. If $n>1$, we can again apply Proposition \ref{back} to get
            \[ac_1(n)=ac_0(n-1) + ac(n-2).\]
            Repeating the same argument $n-2$ more times gives the result. \end{proof}
            
        \begin{prop}\label{sumt} For all $n\geq 0$, we have
            \[\sum_{j=0}^n f_3(j) =\frac{f_3(n)+f_3(n+2)-1}{2}. \]
        \end{prop}
        
        \begin{proof}
            We give a proof by mathematical induction. For $n=0$, \[f_3(0)=0=\frac{f_3(0)+f_3(2)-1}{2}.\]
           Now for $n>0$, suppose the proposition holds for all $k<n$. Then 
           \begin{align*}
               \sum_{j=0}^{n} f_3(j) &= \sum_{j=0}^{n-1} f_3(j) +f_3(n) \\
               &=\frac{f_3(n-1)+f_3(n+1)-1}{2}+f_3(n) \\
               &=\frac{f_3(n+2)-f_3(n)-1}{2} +f_3(n)\\
               &=\frac{f_3(n)+f_3(n+2)-1}{2}.\qedhere
           \end{align*}
        \end{proof}
        
        \begin{prop}\label{lem}
            For all $n\geq 2$, 
            \[ac_1(n)=f_3(n-3)+f_3(n-1).\]
        \end{prop}
        
        \begin{proof}
            By Proposition \ref{sum2}, Theorem 1, and Proposition \ref{sumt} we have
            \begin{align*} ac_1(n)=&\sum_{j=0}^{n-1} ac_0(j)\\
            =& 2\sum_{j=0}^{n-1} f_3(j-2) + ac_0(0)\\
            &=2\sum_{j=0}^{n-3}f_3(j)+ac_0(0) \\
            =&f_3(n-3)+f_3(n-1)-1+ac_0(0).
            \end{align*}
            Since $ac_0(0)=1$, the result follows. \end{proof}
            
            \begin{proof}[Proof of Theorem 2] Theorem 2 now immediately follows from Proposition \ref{lem}, since $ac(1)=1=f_3(1)+f_3(-1)$, and for $n\geq 2$
            \begin{align*}
                ac(n)&=ac_0(n)+ac_1(n) \\
                &=2\cdot f_3(n-2)+f_3(n-3)+f_3(n-1)\\
                &= f_3(n)+f_3(n-2).\qedhere
            \end{align*}
            \end{proof}
    
    \subsection{Proof of Theorem 3 and Corollary 1} In this section we develop the formulae for $ac(n,s)$ by deriving the ordinary generating function $G(q,s)$ for a fixed $s\geq0$. It is easier to split into the cases when $s$ is even and odd.
    
    Suppose $s=2a$, where $a\geq0$. An anti-palindromic composition of $n$ of length $2a$ consists of a sequence of $a$ ordered pairs of distinct positive integers. If $d(n)$ is the number of distinct pairs of positive integers that sum to $n$, then
    \[D(q):=\sum_{n\geq 0}d(n) q^n = \left(\frac{q}{1-q}\right)^2 -\frac{q^2}{1-q^2}=\frac{2q^3}{(1-q^2)(1-q)}.\]
    To see why this is the case, notice that 
    \[\left( \frac{q}{1-q} \right)^2 = \left(q^{1+1}\right) + \left(q^{1+2}+q^{2+1} \right) + \left(q^{1+3} + q^{2+2} + q^{3+1} \right) + \ldots\]
    and
    \[\frac{q^2}{1-q^2} = q^{1+1}+q^{2+2}+q^{3+3}+\ldots,\]
    so we are taking all pairs of positive integers and subtracting the repeated pairs.
    
    To form a sequence of $a$ such pairs, we multiple $D(q)$ by itself $a$ times, showing that
    \[G(q,2a) = [D(q)]^a = \frac{2^aq^{3a}}{(1-q^2)^a(1-q)^a}.\]
    To prove the first half of Corollary 1, recall that for $a>0$
    \[\frac{1}{(1-q^2)^a} = \sum_{n\geq 0} {a-1+n \choose n} q^{2n}\]
    and
    \[\frac{1}{(1-q)^a}=\sum_{n\geq 0}{a-1+n \choose n} q^{n}.\]
    Multiplying these two series and reindexing gives the result.
    
    Now suppose $s=2a+1$, where $a\geq 0$. An anti-palindromic composition of $n$ of length $2a+1$ still consists of $a$ ordered pairs of distinct positive integers, with an additional central part. Therefore,
    \[G(q,2a+1)=G(q,2a)\cdot \frac{q}{1-q}=\frac{2^aq^{3a}}{(1-q^2)^a(1-q)^a}.\]
    The second half of Corollary 1 follows the same way as the first half once we observe
    \[\frac{1}{(1-q)^{a+1}}=\sum_{n\geq 0} {a+n\choose n}q^n.\]
    
    \subsection{Proof of Theorem 4 and Corollary 2} 
    
    We begin with a lemma.
    
    \begin{lemma}\label{rec}
        For $n\geq 3$ and $s\geq 2$ we have
        \[ac(n,s)=ac(n-1,s)+ac(n-2,s)+2\cdot ac(n-3,s-2)-ac(n-3,s).\]
    \end{lemma}
    
    \begin{proof}
        If $\sigma$ is an anti-palindromic composition of $n\geq 3$ of length $s\geq 2$, let $m_{\sigma}:=\sigma_1+\sigma_s$. Observe that $m_{\sigma}\geq 3$ and
        \[\delta(m_{\sigma}) \leq |\sigma_1-\sigma_s|\leq m_{\sigma}-2,\]
        where $\delta(m_{\sigma})=1$ if $m_{\sigma}$ is odd and $\delta(m_{\sigma})=2$ if $m_{\sigma}$ is even.
        
        Let us first count the number of anti-palindromic compositions of $n$ of length $s$ with $m_{\sigma}=3$. Each one of these compositions can be formed by taking an anti-palindromic composition of $n-3$ of length $s-2$ and adjoining a 1 at the beginning and a 2 at the end, or a 2 at the beginning and a 1 at the end. Therefore, the number of anti-palindromic compositions of $n$ of length $s$ with $m_{\sigma}=3$ equals $2\cdot ac(n-3,s-2)$.
        
        Next we count the number of anti-palindromic compositions of $n$ of length $s$ with $m_{\sigma}>3$. Now for any anti-palindromic composition $\tau$ of $n-1$ of length $s$, we can form an anti-palindromic composition of $n$ of length $s$ by adding 1 to $\tau_1$ if $\tau_1>\tau_s$, or adding 1 to $\tau_s$ if $\tau_s>\tau_1$. Now in this way, we have constructed all the anti-palindromic compositions of $n$ of length $s$ with $m_{\sigma}>3$ and $|\sigma_1-\sigma_s|>\delta(m_{\sigma})$.
        
        For any composition $\gamma$ of $n-2$ of length $s$, form an anti-palindromic composition of $n$ of length $s$ by adding 1 to $\gamma_1$ and 1 to $\gamma_s$. In this way, we have constructed all the anti-palindromic compositions of $n$ of length $s$ with $m_{\sigma}>3$ and $|\sigma_1-\sigma_s|\leq m_{\sigma}-4$. 
        
        Therefore, the total number of anti-palindromic compositions of $n$ of length $s$ with $m_{\sigma}>3$ and $\delta(m_{\sigma})\leq |\sigma_1-\sigma_s|\leq m_{\sigma}-2$ equals $apc(n-1,s)+apc(n-2,s)$ minus the anti-palindromic compositions of $n$ of length $s$ with $m_{\sigma}>3$ and $\delta(m_{\sigma})<|\sigma_1-\sigma_s|\leq m_{\sigma}-4$, as we have counted these compositions exactly twice. To prove the lemma, we now must show that the number of compositions that we counted twice equals $apc(n-3,s)$.
        
        Let $\rho$ be an anti-palindromic composition of $n-3$ of length $s$. Form an anti-palindromic composition of $n$ of length $s$ by adding 2 to $\rho_1$ and 1 to $\rho_s$ if $\rho_1>\rho_s$, or 1 to $\rho_1$ and 2 to $\rho_s$ if $\rho_s>\rho_1$. In this way, we have constructed all of the anti-palindromic compositions of $n$ of length $s$ with $\delta(m_{\sigma})<|\sigma_1-\sigma_s|\leq m_{\sigma}-4$.
    \end{proof}
    
    \begin{proof}[Proof of Theorem 4]
    
    The theorem can be verified for all $n$ and $s$ with $n+s<5$:

    \begin{align*}
        \phi_0(q)&=rac(0,0)\cdot q^0+rac(0,1)\cdot q^1+rac(0,2)\cdot q^2 = 1\cdot q^0+0\cdot q^1+ 0\cdot q^2 \\
        \phi_1(q)&=rac(1,0)\cdot q^0+rac(1,1)\cdot q^1+rac(1,2)\cdot q^2 = 0\cdot q^0+1\cdot q^1+ 0\cdot q^2\\
        \phi_2(q)&=rac(2,0)\cdot q^0+rac(2,1)\cdot q^1+rac(2,2)\cdot q^2 = 0\cdot q^0+1\cdot q^1 + 0\cdot q^2\\
        \phi_3(q)&=rac(3,0)\cdot q^0+rac(3,1)\cdot q^1+rac(3,2)\cdot q^2 = 0\cdot q^0+1\cdot q^1 + 1\cdot q^2.
   \end{align*}

    Let $[q^s]\phi_n(s)$ be the coefficient of $q^s$ in the polynomial $\phi_n(s)$. Now for $n\geq3$ and $s\geq 2$, using the defining recurrence for $\phi_n(q)$ we have
    \begin{align*}
    [q^s]\phi_n(q)&=[q^s]\phi_{n-1}(q) + [q^s]\phi_{n-2}(q) +[q^{s-2}]\phi_{n-3}(q) - [q^s]\phi_{n-3}(q) \\
    &=rac(n-1,s)+rac(n-2,s)+ rac(n-3,s-2)-rac(n-3,s)
    \end{align*}
    by induction. Using the relationship between $rac(n,s)$ and $ac(n,s)$,
    \begin{align*}
    [q^s]\phi_n(q)&=\frac{ac(n-1,s)}{2^{\lfloor \frac{s}{2} \rfloor}}+\frac{ac(n-2,s)}{2^{\lfloor \frac{s}{2} \rfloor}}+\frac{ ac(n-3,s-2)}{2^{\lfloor \frac{s-2}{2} \rfloor}}-\frac{ac(n-3,s)}{2^{\lfloor \frac{s}{2} \rfloor}}\\
    &=\frac{ac(n-1,s)}{2^{\lfloor \frac{s}{2} \rfloor}}+\frac{ac(n-2,s)}{2^{\lfloor \frac{s}{2} \rfloor}}+\frac{ 2\cdot ac(n-3,s-2)}{2^{\lfloor \frac{s}{2} \rfloor}}-\frac{ac(n-3,s)}{2^{\lfloor \frac{s}{2} \rfloor}} \\
    &= \frac{ac(n,s)}{2^{\lfloor \frac{s}{2} \rfloor}}
    \end{align*}
    by Lemma \ref{rec}. Therefore, $[q^s]\phi_n(s) = rac(n,s)$. \end{proof}
    
    \begin{proof}[Proof of Corollary 2]
    Notice that by Theorem 4 we have \[rac(n)=\sum_{s\geq 0} rac(n,s)=\phi_n(1) = f_2(n).\]
    As for $rac_0(n)$, we have $rac_0(1)=rac_0(2)=0$, $rac_0(3)=1$, and for $n\geq 4$
    \[rac_0(n)=\sum_{s\geq 0} rac(n,2s) = \frac{\phi_n(1)+\phi_n(-1)}{2},\]
    again using Theorem 4.
    By the definition of $\phi_n(q)$, this equals
    \[\frac{\phi_{n-1}(1)+\phi_{n-1}(-1)}{2}+\frac{\phi_{n-2}(1)+\phi_{n-2}(-1)}{2}=rac_0(n-1)+rac_0(n-2).\]
    This is the defining recurrence relation for the Fibonacci numbers, thus we conclude that $rac_0(n)=f_2(n-2)$.
    
    Similarly for $rac_1(n)$, we have $rac_1(0)=1$, $rac_1(2)=rac_1(3)=1$, and for $n\geq 4$
    \[rac_1(n)=\sum_{s\geq 0} rac(n,2s+1) = \frac{\phi_n(1)-\phi_n(-1)}{2}\]
    by Theorem 4. 
    By the definition of $\phi_n(q)$, this equals
    \[\frac{\phi_{n-1}(1)-\phi_{n-1}(-1)}{2}+\frac{\phi_{n-2}(1)-\phi_{n-2}(-1)}{2}=rac_1(n-1)+rac_1(n-2).\]
    This is the defining recurrence relation for the Fibonacci numbers, thus we conclude that $rac_1(n)=f_2(n-1)$. \end{proof}
    
    \section*{Acknowledgements}
    
    The second author (M.J.) was partially supported by the Research and Training Group grant DMS-1344994 funded by the National Science Foundation. We thank Robert Schneider and Drew Sills for helpful comments.


\begin{thebibliography}{10}\footnotesize
    
   \bibitem{A04} G. E. Andrews, Fibonacci numbers and the Rogers–Ramanujan identities, {\it Fibonacci Quart.} {\bf 42} (2004), 3-19.
   
  
    
    \bibitem{AH75} K. Alladi and V. E. Hoggatt, Jr., Compositions with ones and twos, {\it Fibonacci Quart.} {\bf 13} (1975), 233-239.
    
    \bibitem{AS20} G. Andrews and G. Simay, Parity Palindrome Compositions, submitted
    
    \bibitem{BCSS11} A. T. Benjamin, P. Chinn, J. N. Scott, and G. Simay, Combinatorics of two-toned tilings, {\it Fibonacci Quart.} {\bf 49} (2011).
    
    \bibitem{BL13} J. J. Bravo and F. Luca, Coincidences in generalized Fibonacci sequences, {\it J. Number Theory} {\bf 6} (2013), 2121-2137.
    
    \bibitem{C75} L. Carlitz, Fibonacci notes. III: $q$-Fibonacci numbers, {\it Fibonacci Quart.} {\bf 12} (1974), 317–322.
    
    \bibitem{C03} J. Cigler, $q$-Fibonacci polynomials, {\it Fibonacci Quart.} {\bf 41} (2003), 31-40.
    
    \bibitem{Eg13} S. Eger, Restricted weighted integer compositions and extended binomial coefficients,  {\it Journal of Integer Sequences}, {\bf 16} (2013).
    
    \bibitem{HM04} S. Heubach and T. Mansour, Compositions of $n$ with parts in a set, {\it Congressus Numerantium.} {\bf 168} (2004), 33–51.
    
    \bibitem{comp} S. Heubach and T. Mansour, {\it Combinatorics of Compositions and Words}, CRC Press, 2010.
    
     
    \bibitem{HB75} V. E. Hoggatt, Jr. and M. Bicknell, Palindromic compositions, {\it Fibonacci Quart.}, {\bf 13(4)} (1975), 350-356.
    
    \bibitem{J20} M. Just, Compositions that are palindromic modulo $m$, submitted.
    
    \bibitem{eleph} \'A. Kun, J. Podani and J. Szil\'agyi, How fast does Darwin’s elephant population grow?, {\it J Hist Biol}  {\bf 51} (2018), 259-281.
    
    \bibitem{P06} H. Pan, Arithmetic properties of $q$-Fibonacci numbers and $q$-pell numbers, {\it Discrete Mathematics} {\bf 306} (2006), 2118-2127.
    
   
    
     \bibitem{trib} OEIS Foundation Inc. (2020), The On-Line Encyclopedia of Integer Sequences, http://oeis.org/A000073.

    
    \end{thebibliography}
\end{document}